\documentclass[a4paper]{amsart}
\usepackage[utf8]{inputenc}
\usepackage[a4paper,margin=2.5cm]{geometry}
\usepackage{amsmath,amssymb,amsthm,mathtools,abraces}

\usepackage{color,xcolor}
\usepackage[hypertexnames=false,hyperfootnotes=false,colorlinks=true,linkcolor=blue,%
citecolor=purple,filecolor=magenta,urlcolor=cyan,unicode,linktocpage=true]{hyperref}

\newtheorem{theorem}{Theorem}[section]     
\newtheorem{proposition}[theorem]{Proposition} 
\newtheorem{lemma}[theorem]{Lemma} 
\newtheorem{corollary}[theorem]{Corollary}

\theoremstyle{definition}

\theoremstyle{remark}
\newtheorem{remark}[theorem]{Remark}

\newcommand{\Mgn}{\overline{\mathcal{M}}_{g,n}}
\newcommand{\I}{\mathsf{i}}
\newcommand{\sgn}{\mathop{\mathrm{sgn}}}
\newcommand{\id}{{\mathrm{id}}}

\usepackage{filecontents}

\title[Buryak-Okounkov formula for the $n$-point function]{Buryak-Okounkov formula for the $n$-point function and a new proof of the Witten conjecture}

\author{Alexander Alexandrov}

\address[A. Alexandrov]{Center for Geometry and Physics, Institute for Basic Science (IBS), Pohang 37673, Korea; and Institute for Theoretical and Experimental Physics, 25 Bolshaya Cheryomushkinskaya
Ulitsa, Moscow 117218, Russia}

\email{alex@ibs.re.kr}

\author{Francisco Hern{\'a}ndez Iglesias}

\address[F. Hern{\'a}ndez Iglesias]{Korteweg-de Vriesinstituut voor Wiskunde, 
	Universiteit van Amsterdam, Postbus 94248,
	1090GE Amsterdam, Nederland}
\email{f.hernandeziglesias@uva.nl}

\author{Sergey Shadrin}

\address[S. Shadrin]{Korteweg-de Vriesinstituut voor Wiskunde, 
	Universiteit van Amsterdam, Postbus 94248,
	1090GE Amsterdam, Nederland}
\email{s.shadrin@uva.nl}

\begin{document}

\begin{abstract} We identify the formulas of Buryak and Okounkov for the $n$-point functions of the intersection numbers of psi-classes on the moduli spaces of curves. This allows us to combine the earlier known results and this one into a principally new proof of the famous Witten conjecture / Kontsevich theorem, where the link between the intersection theory of the moduli spaces and integrable systems is established via the geometry of double ramification cycles. 
\end{abstract}

\maketitle

\tableofcontents

\section{Introduction}

The symbol $\langle \prod_{i=1}^n \tau_{d_i} \rangle _g$ denotes the intersection number $\int_{\Mgn} \prod_{j=1}^n\psi_j^{d_j}$. It can be non-zero only if $g\geq 0$, $n\geq 1$, $2g-2+n>0$, $d_1,\dots,d_n\geq 0$, and $\sum_{j=1}^n d_j = 3g-3+n = \dim \Mgn$. Witten conjectured~\cite{WittenConj} that the generating function of these intersection numbers defined as
\[
F:=\sum_{g=0}^n \left\langle \exp\left(\sum_{d=0}^{\infty} \tau_d t_d \right) \right\rangle_g
\]
is the logarithm of the unique tau-function of the Korteweg-de Vries (KdV) hierachy that in addition satisfies the string equation,
\[
\left[\frac{\partial}{\partial t_0} -\sum_{d=0}^\infty t_{d+1} \frac{\partial}{\partial t_d} - \frac{t_0^2}{2}\right] \exp(F) = 0.
\]
The string equation is easy to prove, see~\cite{WittenConj}, so the main part of the conjecture is the equations of the KdV hierarchy. 

This conjecture was first proved by Kontsevich in~\cite{Kontsevich} using the Strebel-Penner ribbon graph model of the moduli space of curves, and later on more proofs have appeared. Mirzakhani~\cite{mirzakhani} used symplectic reduction for the Weil-Peterson volumes of the moduli space, and Okounkov and Pandharipande~\cite{okounkovpand} and Kazarian and Lando~\cite{kazarianlando} used the ELSV formula that connects the intersection theory and Hurwitz numbers. There are more papers where the Witten conjecture / Kontsevich theorem is proved (see e.~g.~\cite{OkounkovMain,MulaseSafnuk,Kazarian,KimLiu,ChenLiLiu,Witten}), 
but on the geometric side they all use one of the ideas mentioned above: the Strebel-Penner ribbon graph model, symplectic reduction, or the ELSV formula for Hurwitz numbers.

In this paper we give a new proof of the Witten conjecture based on a completely different geometric idea than any of the earlier existing proofs: the intersection theory of \emph{double ramification cycles}. More precisely, the full proof that we explain here consists of four big steps, where three of them were already available in the literature, and the fourth missing one is the main subject of this paper:
\begin{enumerate}
\item In~\cite{BSSZ} Buryak \textit{et~al.}~fully described the intersection numbers of the monomials of psi-classes with the double ramification cycles.
\item In~\cite{BuryakMain} Buryak  used the previous result and a relation between the double ramification cycles and the fundamental cycles of the moduli spaces of curves to describe explicitly the so-called $n$-point function $\mathcal{F}_{n}=\mathcal{F}_{n}(x_1,\dots,x_n):=\sum_{g\geq 0}\sum_{d_1,\dots,d_n\geq 0} \langle \prod_{i=1}^n \tau_{d_i} x_i^{d_i} \rangle _g$, $n\geq 1$. 
\item In~\cite{OkounkovMain} Okounkov proved a different explicit formula for
the $n$-point functions $\mathcal{F}_{n}$  and he showed in Section 3 of \emph{op.~cit.}~that the generating function of
their coefficients is the logarithm of the string tau-function of the
KdV hierarchy.
\item In this paper we identify Buryak's and Okounkov's formulas for the $n$-point function, and this makes the sequence of papers \cite{BSSZ} $\to$ \cite{BuryakMain} $\to$ {the present paper} $\to$ \cite[Section 3]{OkounkovMain} a new proof of the Witten conjecture.  
\end{enumerate}

Let us say a few words about the geometric techniques used in~\cite{BSSZ} and~\cite{BuryakMain}. A double ramification cycle $DR_g(a_1,\dots,a_n)$, $a_i\in\mathbb{Z}$, $\sum_{i=1}^n a_i = 0$ is the class of a certain compactification of the locus of the isomorphism classes of smooth curves with marked points $(C_g,x_1,\dots,x_n\in C_g)\in\Mgn$ such that $\sum_{i=1}^n{a_ix_i}$ is the divisor of a meromorphic function $C_g\to\mathbb{C}\mathrm{P}^1$. These cycles inherit rich geometry of the space of maps to $\mathbb{C}\mathrm{P}^1$ and this allows to express the psi-classes restricted to these cycles in terms of the double ramification cycles of smaller dimension, which is in principle enough to compute all intersection numbers of psi-classes with the double ramification cycles. Next, observe that under the projection $\overline{\mathcal{M}}_{g,n+g}\to \Mgn$ that forgets $g$ marked points the push-forward of a double ramification cycle is a multiple of the fundamental cycle of $\Mgn$. This relates the intersection numbers of psi-classes on double ramification cycles to $\langle \prod_{i=1}^n \tau_{d_i} \rangle _g$. There is, of course, a long way from these computational ideas to nice closed formulas derived in~\cite{BSSZ} and~\cite{BuryakMain}.

Let us stress that in the approach of~\cite[Section 3.2]{BuryakMain} it is sufficient to assume that all weights of marked points in double ramification cycles are non-zero integers (for instance, assume that all integer numbers chosen arbitrarily in the beginning of the argument of Buryak are positive). This allows to use only part of the results of~\cite{BSSZ} that concerns the intersection numbers of psi-classes with the double ramification cycles with only non-zero weights. This part of the computation in~\cite{BSSZ} uses nothing but the factorization rules for psi-classes at the points of non-zero weights on double ramification cycles, which work equally well for the double ramification cycles defined via relative stable maps to $\mathbb{C}\mathrm{P}^1$ and the double ramification cycles of admissible covers~\cite{Ionel} (cf.~a discussion in~\cite[Section 2.3]{BSSZ}).

This idea of computation of the intersection numbers has been used in a number of earlier papers, cf.~\cite{ShaUMN03,ShaIMRN03,Sha05,Sha06,Bursha}, and these papers might serve a good source of examples of particular computations. In particular, an explicit algorithm for the computation of all intersection numbers $\langle \prod_{i=1}^n \tau_{d_i} \rangle _g$ is given in~\cite{ShaZvo08}.  
%
%
Exactly the same idea of computation of the intersection numbers of $\psi$-classes is proposed in~\cite[Section 9]{costello}. It is mentioned in~\cite[Section 1.3]{costello} that for the further applications of the results of that paper a required first step is to give a new proof of Witten's conjecture~\cite{WittenConj} using the technique developed there. So, it is precisely what the present paper (combined with~\cite{BSSZ}, \cite{BuryakMain}, and \cite{OkounkovMain}) does.

Finally, to conclude the introduction, let us mention that the $n$-point functions for the intersection numbers of psi-classes have recently been studied from different points of view, see~\cite{EOr,LiuXu,Zhou,bertoladubyang,ZhouEmergent,BH0704,BH0709,AMMP}. The comparison of different formulas and recursive relations for their coefficients is very interesting and usually highly non-trivial, and this paper can also be considered as a step towards unification (see also~\cite{ZhouEmergent}) of the variety of formulas for the $n$-point functions.

\subsection{Organization of the paper} In Section~\ref{sec:BuryakOkounkov} we recall the formulas of Buryak and Okounkov and some statements about these formulas that we use in this paper, and state our main results. In Section~\ref{sec:BuryakRev} we derive an equivalent form of the Buryak formula. In Section~\ref{sec:Principal} we prove that the principal terms in  Buryak and Okounkov formulas coincide. In Section~\ref{sec:diagonal} we prove that all other terms, namely, the so-called diagonal terms needed for a regularization of the principal ones, also coincide in Buryak and Okounkov formulas. 

\subsection{Acknowledgments} We thank A.~Buryak, G.~Carlet, R.~Kramer, D.~Maulik, and A.~Okounkov for useful discussions and correspondence. A.~A. was supported by IBS-R003-D1 and by RFBR grant 17-01-00585. 
 F.~H.~I. and S.~S. were supported by the Netherlands Organization for Scientific Research.  A.A. wishes to thank the KdV Institute for its kind hospitality. 

\section{Buryak and Okounkov formulas} \label{sec:BuryakOkounkov} In this section we recall the formulas for the $n$-point functions in~\cite{BuryakMain} and~\cite{OkounkovMain}. It is convenient to append the intersection numbers by two unstable cases $g=0$ and $n=1,2$. Namely, we assume by definition that $\langle \sum_{d_1\geq 0}\tau_{d_1} x_1^{d_1} \rangle_0 \coloneqq x_1^{-2}$ and $\langle \sum_{d_1,d_2\geq 0}\tau_{d_1} x_1^{d_1} \tau_{d_2} x_2^{d_2} \rangle_0 \coloneqq (x_1+x_2)^{-1}$, and we add these terms to $\mathcal{F}_1$ and $\mathcal{F}_2$, respectively.

\subsection{Formula of Buryak} Let $\zeta(x) \coloneqq e^{x/2} - e^{-x/2}$. Define the function $P_n(a_1,\dots,a_n;x_1,\dots,x_n)$ by $P_1(a_1;x_1) \coloneqq \frac{1}{x_1}$ and for $n\geq 2$ we have
\begin{align}\label{eq:DefinitionP}
P_n(a_1, \dots, a_n; x_1, \dots, x_n) \coloneqq \sum_{\substack{\tau \in S_n\\ \tau(1) = 1 }}
\frac{ 
\prod\limits_{j=2}^{n-1} x_{\tau(j)} 
\prod\limits_{j=1}^{n-1}
\zeta \left( 
	\left(\sum\limits_{k=1}^{j} a_{\tau(k)}\right)  x_{\tau(j+1)} -  a_{\tau(j+1)} \left( \sum\limits_{k=1}^{j} x_{\tau(k)} \right)
 \right) 
} {
	\prod\limits_{j=1}^{n-1} \left(
	a_{\tau(j)}x_{\tau(j+1)} - a_{\tau(j+1)} x_{\tau(j)} 
	\right)
}.
\end{align}
Though it is not obvious from the definition, $P_n$ is a formal power series in all its variables, which is invariant with respect to the diagonal action of the symmetric group $S_n$ on $(a_1,\dots,a_n)$ and $(x_1,\dots,x_n)$, see~\cite[Remarks 1.5 and 1.6]{BSSZ}.

Define the function $\mathcal{F}_n^{\mathsf{Bur}}=\mathcal{F}_n^{\mathsf{Bur}}(x_1,\dots,x_n)$ as the Gaussian integral
\begin{align*}
\mathcal{F}_n^{\mathsf{Bur}}(x_1,\dots,x_n) \coloneqq 
\frac{e^{ \left( \sum_{j=1}^{n} x_j \right)^3/24 }}{\left( \sum_{j=1}^{n} x_j \right) \prod_{j=1}^{n} \sqrt{2 \pi x_j} } \int_{\mathbb{R}^n} \left[\prod_{j=1}^{n} e^{ - \frac{a_j^2}{2 x_j} } da_j \right] P_n(\I a_1, \dots, \I a_n; x_1, \dots, x_n).
\end{align*}

\begin{theorem}[Buryak~\cite{BuryakMain}] \label{thm:buryak}
	For $n\geq 1$ we have $\mathcal{F}_n = \mathcal{F}_n^{\mathsf{Bur}}$.
\end{theorem}

\subsection{Formula of Okounkov} Define the function $\mathcal{E}(x_1,\dots,x_n)$ as
\begin{align*}
 \mathcal{E}(x_1,\dots,x_n) \coloneqq \frac{e^{ \left( \sum_{j=1}^{n} x_j^3 \right)/12 }}{\prod_{j=1}^{n} \sqrt{4\pi x_j} } \int_{\mathbb{R}_{\geq 0}^n} \left[\prod_{j=1}^{n} ds_j \right]
 \exp \left({-\sum_{j=1}^n \left(\frac{(s_j-s_{j+1})^2}{4x_j} + \frac{(s_j+s_{j+1})x_j}{2} \right)}\right),
\end{align*}
where $s_{n+1}$ denotes $s_1$. Then the function $\mathcal{E}^{\circlearrowleft}(x_1,\dots,x_n)$ defined as
\begin{align*}
\mathcal{E}^{\circlearrowleft}(x_1,\dots,x_n) \coloneqq \sum_{\sigma\in S_n/\mathbb{Z}_n} 
\mathcal{E}(x_{\sigma(1)},\dots,x_{\sigma(n)})
\end{align*}
is invariant under the $S_n$-action on $(x_1,\dots,x_n)$. 

Denote by $\Pi_n$ the set of all partitions of the set $\{1,\dots,n\}$ into a disjoint union of unordered subsets $\sqcup_{j=1}^\ell I_j$, for all $\ell=1,2\dots,n$. Let $x_I\coloneqq \sum_{j\in I} x_j$, $I\subset \{1,\dots,n\}$. Define the function $\mathcal{G}(x_1,\dots,x_n)$ as
\begin{align*}
\mathcal{G}(x_1,\dots,x_n)\coloneqq \sum_{\sqcup_{j=1}^\ell I_j \in \Pi_n} (-1)^{\ell+1} \mathcal{E}^{\circlearrowleft}(x_{I_1},\dots,x_{I_\ell})
\end{align*}
and the function $\mathcal{F}_n^{\mathsf{Ok}}=\mathcal{F}_n^{\mathsf{Ok}}(x_1,\dots,x_n)$ as
\begin{align*}
\mathcal{F}_n^{\mathsf{Ok}}(x_1,\dots,x_n)\coloneqq  \frac{(2\pi)^{n/2}}{\prod_{j=1}^n\sqrt{x_j}}
\mathcal{G}\left(\frac{x_1}{2^{1/3}},\dots,\frac{x_n}{2^{1/3}}\right).
\end{align*}

\begin{theorem}[Okounkov~\cite{OkounkovMain}] \label{thm:okounkov}
The generating function of the coefficients of $\mathcal{F}_n^{\mathsf{Ok}}$, $n\geq 1$, is the logarithm of the string tau-function of the KdV hierarchy. 
\end{theorem}

\subsection{Main theorem} We are ready to state our main result.

\begin{theorem} \label{thm:main} We have: $\mathcal{F}_n^{\mathsf{Bur}} = \mathcal{F}_n^{\mathsf{Ok}}$, $n\geq 1$.  
\end{theorem}

The rest of the paper is devoted to the proof of this theorem. An immediate corollary of Theorems~\ref{thm:buryak},~\ref{thm:okounkov}, and~\ref{thm:main} is the following:

\begin{corollary} The Witten conjecture is true, that is, the function $\exp(F)$ is the string tau-function of the KdV hierarchy. 	
\end{corollary}

As we explain in the introduction, the real importance of this new proof of the Witten conjecture is that it uses a new way to relate the intersection theory of the moduli space of curves to the theory of integrable hierarchies, based on geometry of double ramification cycles. Otherwise, though Theorem~\ref{thm:main} is interesting by itself, the identity $\mathcal{F}_n = \mathcal{F}_n^{\mathsf{Ok}}$ has an alternative proof in~\cite[Section 2]{OkounkovMain}. 

\section{Buryak formula revisited} \label{sec:BuryakRev} Our first goal is to translate the cumbersome formula of Buryak into something more manageable. Let $w_{jk}\coloneqq (a_jx_k-a_kx_j)/2$ and $u_{jk}\coloneqq a_j/x_j-a_k/x_k$. 

\begin{proposition}\label{prop:newP} For $n\geq 1$ we have:
\begin{align} \label{eq:NewP}
P_n(a_1, \dots, a_n; x_1, \dots, x_n) = \frac{1}{\prod_{i=1}^{n} x_i} \sum_{\sigma \in S_n} \frac{\exp \left( \sum_{i < j} w_{\sigma(i) \sigma(j)} \right)}{\prod_{j=1}^{n-1} u_{\sigma(j) \sigma(j+1)} } 
\end{align}
\end{proposition}

It is clearly true for $n = 1$ and we prove it below for $n \geq 2$. Now the function $P_n$ is manifestly invariant with respect to the diagonal action of the symmetric group $S_n$ on $(a_1,\dots,a_n)$ and $(x_1,\dots,x_n)$. 

\begin{corollary} We have:
\begin{align}\label{eq:buryak-symmetric}
\mathcal{F}_n^{\mathsf{Bur}}=
\frac{e^{ \frac{1}{24}\left( \sum\limits_{j=1}^{n} x_j \right)^3 }}{\left( \sum\limits_{j=1}^{n} x_j \right) (2 \pi)^{\frac n2} \prod_{j=1}^{n}  x_j^{\frac 32} } \int_{\mathbb{R}^n} \left[\prod_{j=1}^{n} e^{ - \frac{a_j^2}{2 x_j} } da_j \right] \sum_{\sigma \in S_n} \frac{\exp \left( \frac{\I}{2} \sum_{j < k} 
a_{\sigma(j)}x_{\sigma(k)}-a_{\sigma(k)}x_{\sigma(j)}
\right)}{\prod_{j=1}^{n-1} \I \left( \frac{a_{\sigma(j)}}{x_{\sigma(j)}} - \frac{a_{\sigma(j+1)}}{x_{\sigma(j+1)}}\right)  } .
\end{align}
\end{corollary}

\subsection{Proof of Proposition~\ref{prop:newP}} Assume that $n\geq 2$. Expanding the definition of the function $\zeta$ allows us to rewrite Equation~\eqref{eq:DefinitionP} for $P_n=P_n(a_1, \dots, a_n; x_1, \dots, x_n)$ as

\begin{equation} \label{eqn:functionpartition}
P_n = \frac{1}{\prod_{i=1}^{n} x_i} \sum_{\substack{\tau \in S_n\\ \tau(1) = 1 }} \sum_{\substack{I \sqcup J= \\ \{2,\dots,n\}}}  \frac{
		(-1)^{|J|} \exp 
		\left( \sum_{i \in I} \sum_{\ell = 1}^{i-1} w_{ \tau(\ell) \tau(i) } - \sum_{j \in J} \sum_{\ell = 1}^{j-1} w_{ \tau(\ell) \tau(j) }  
		\right) 
	}{
		\prod_{j=1}^{n-1} u_{\tau(j) \tau(j+1)} 
	} 
\end{equation}

\subsubsection{Exponential terms in the numerators}

In order to identify Equations~\eqref{eq:NewP} and~\eqref{eqn:functionpartition}, we consider for each particular fixed sequence of signs $\sgn(w_{rs})=\pm 1$, $r<s$, all terms in Equations~\eqref{eq:NewP} and~\eqref{eqn:functionpartition} where the numerator is equal to $\exp(A)$, $A=\sum_{r<s} \sgn(w_{rs}) w_{rs}$, and prove that the total coefficient of $\exp(A)$ coincides in both formulas. The symbols $w_{ij}$, $1\leq i,j\leq n$, are understood in the rest of the proof as just formal variables satisfying the relations $w_{ij} + w_{ji} = 0$.

Let $[2,n]$ denote the set $\{2,\dots,n\}$. For $\sigma\in S_n$ and $I\sqcup J = [2,n]$ we define 
\begin{equation*}
A^\sigma_{I,J} \coloneqq \sum_{i\in I}\sum_{\ell=1}^{i-1} w_{\sigma(\ell)\sigma(i)} -
\sum_{j\in J}\sum_{\ell=1}^{j-1} w_{\sigma(\ell)\sigma(j)}.
\end{equation*}
It is a convenient way to keep track of signs in the exponential terms in the numerators of~\eqref{eq:NewP} and~\eqref{eqn:functionpartition}. It is easy to see that 
\begin{itemize}
	\item In Equation~\eqref{eq:NewP} the numerators are indexed by  $\exp(A^\sigma_{[2,n],\emptyset})$, for all $\sigma \in S_n$;
	\item In Equation~\eqref{eqn:functionpartition} the numerators are indexed by 
	$\exp(A^\tau_{I,J})$, for all $\tau \in S_n$ such that $\tau(1)=1$ and for all $I\sqcup J = [2,n]$.
\end{itemize}
So, we have to obtain a full description of all $\sigma,\tau$, and $I \sqcup J$ as above such that $\exp(A^\sigma_{[2,n],\emptyset}) = \exp(A^\tau_{I,J})$.

\subsubsection{Notation for the symmetric group} Decompose $S_n$ as $S_{n-1}\sqcup\left(\sqcup_{i=2}^n S_{n-1} (1,i)\right)$, where $S_{n-1}\subset S_n$ denotes the subgroup of permutations $\tau$ such that $\tau(1)=1$. 

Denote by $C_m$, $m\geq 2$, the cyclic permutation $(1,m,m-1,\dots,2)$. Consider the subset 
$T\subset S_n$ defined as $T\coloneqq \{\id\}\cup \left(\cup_{i=1}^{n-1} \{ C_{m_1}\cdots C_{m_i}\ |\ 2\leq m_1<\cdots <m_i\leq n  \} \right)$. The following lemma implies that it is in fact a disjoint union.

\begin{lemma} We have: $T\cap S_{n-1} = \{\id\}$, and 
\[	
	 \quad T \cap (S_{n-1} (1,i)) = \{C_{m_1}\cdots C_{m_{i-1}}\ |\ 2\leq m_1<\cdots <m_{i-1}\leq n  \}, \quad i \geq 2.
\]
\end{lemma}

\begin{proof} Observe that $T=(T\cap S_{n-1})\sqcup\left(\sqcup_{i=2}^n T \cap (S_{n-1} (1,i))\right)$. Hence it is enough to show that $\{\id\}\subset S_{n-1}$ (which is obvious) and  $\{C_{m_1}\cdots C_{m_i}\ |\ 2\leq m_1<\cdots <m_i\leq n  \}  \subset (S_{n-1} (1,i))$, $i\geq 2$.

The latter fact we can prove by induction. 	For $i=2$ we see that $C_m = (2,m,m-1,\dots,3) (1,2)$. Assume we know that for any $2\leq m_1<\cdots <m_{i-1}\leq n$ the product $C_{m_1}\cdots C_{m_{i-1}}$ is equal to $\tau (1,i)$ for some $\tau\in S_{n-1}$. Then for any $2\leq m_1<\cdots <m_{i}\leq n$ we have:
 \begin{align*}
& C_{m_1}\cdots C_{m_i} = \tau (1,i) C_{m_i} 
 =\tau (2,i,i-1,\dots,3)(i+1,m_i,m_i-1,\dots,i+2)(1,i+1)=\tau'(1,i+1),
 \end{align*}
where $\tau'\in S_{n-1}$. Thus $C_{m_1}\cdots C_{m_i} \in S_{n-1}  (1,i+1)$.
\end{proof}

\subsubsection{$A^\sigma_{[2,n],\emptyset}$ versus $A^\tau_{I,J}$} The full description of the correspondences between $A^\sigma_{[2,n],\emptyset}$, $\sigma\in S_n$, and $A^\tau_{I,J}$, $\tau\in S_{n-1}$, $I\sqcup J=[2,n]$, is given by the following lemma.

\begin{lemma} \label{lem:A-A} (1) For any $\tau\in S_{n-1}$, $I\sqcup J=[2,n]$, there exists a $\sigma\in S_n$ such that $A^\sigma_{[2,n],\emptyset}=A^\tau_{I,J}$.
	
(2) For any $\sigma\in S_{n-1}$ the only combination  of $(\tau,I,J)$, where $\tau\in S_{n-1}$ and $I\sqcup J=[2,n]$, such that $A^\tau_{I,J}=A^\sigma_{[2,n],\emptyset}$ is given by $\tau=\sigma$, $I=[2,n]$, $J=\emptyset$. 

(3) For any $\sigma\in S_{n-1}(1,i)$, $i\geq 2$, the complete list of the combinations $(\tau,I,J)$, where $\tau\in S_{n-1}$ and $I\sqcup J=[2,n]$, such that $A^\tau_{I,J}=A^\sigma_{[2,n],\emptyset}$ is indexed by the sequences $2\leq m_1<\cdots m_{i-1} \leq n$, where
\[
\tau = \sigma C_{m_{i-1}}^{-1} \cdots C_{m_1}^{-1}; \quad I=[2,n]\setminus \{m_1,\dots,m_{i-1}\}; \quad J=\{m_1,\dots,m_{i-1}\}.
\]
\end{lemma}

\subsubsection{Comparison of the coefficients} The symbols $u_{ij}$, $1\leq i,j\leq n$, are understood in the rest of the proof as just formal variables satisfying the relations $u_{ij} + u_{ji} = 0$ and $u_{ij}+u_{jk}+u_{ki}=0$ for all $i,j,k$. For $\sigma \in S_n$, $n\geq 2$, the symbols $Q(\sigma)$ denotes
\[
Q(\sigma) \coloneqq \frac{1}{u_{\sigma(1)\sigma(2) } u_{\sigma(2)\sigma(3) } \dots u_{\sigma(n-1)\sigma(n) } }.
\]

Up to a factor $1/\prod_{i=1}^n x_i$ (which is a common factor for~\eqref{eq:NewP} and~\eqref{eqn:functionpartition}), the coefficient of $\exp(A^\sigma_{[2,n],\emptyset})$ in~\eqref{eq:NewP} is equal to $Q(\sigma)$. Up to the same factor, the coefficient of $\exp(A^\tau_{I,J})$ is equal to $(-1)^{|J|} Q(\tau)$.
	
\begin{lemma} \label{lem:coeffP} For any $\sigma\in S_{n-1}(1,i)$, $2\leq i\leq n$, we have:
	\begin{equation}\label{eq:idenQ}
		Q(\sigma) = (-1)^{i-1} \sum_{2\leq m_1<\cdots < m_{i-1}\leq n} Q(\sigma C_{m_{i-1}}^{-1} \cdots C_{m_1}^{-1}).
	\end{equation}
\end{lemma}

Lemma~\ref{lem:A-A} and Lemma~\ref{lem:coeffP} together imply that the right hand side of Equation~\eqref{eq:NewP} is equal to the right hand side of Equation~\eqref{eqn:functionpartition}, which completes the proof of Proposition~\ref{prop:newP}.

\subsection{Technical lemmas} In this section we prove Lemma~\ref{lem:A-A} and Lemma~\ref{lem:coeffP} used in the proof of Proposition~\ref{prop:newP}.

\subsubsection{Proof of Lemma~\ref{lem:A-A}} The proof is based on several observations. 
First, observe the left invariance of the identities for $A^\sigma_{I,J}$:
\begin{lemma}\label{lem:leftmultA} We have: $A^{\sigma}_{[2,n],\emptyset}=A^{\id}_{I, J}$ implies $A^{\rho\sigma}_{[2,n],\emptyset}=A^{\rho}_{I, J}$ for any $\rho\in S_n$.
\end{lemma}

\begin{proof}
	Direct inspection of signs.
\end{proof}

Second, we have uniqueness:
\begin{lemma} \label{lem:uniqueA}
	The equality $A^\sigma_{[2,n],\emptyset}=A^{\id}_{I,J}$ considered as an equation for $\sigma$ has at most one solution.
\end{lemma}

\begin{proof}  Assume we have two solutions, $\sigma$ and $\rho$, that is, $A^\sigma_{[2,n],\emptyset}=A^{\id}_{I,J}=A^\rho_{[2,n],\emptyset}$. Applying Lemma~\ref{lem:leftmultA} twice, we obtain: $A^{\rho^{-1}\sigma}_{[2,n],\emptyset} = A^{\rho^{-1}}_{I,J} = A^{\id}_{[2,n],\emptyset}$. Hence $\rho^{-1}\sigma = \id$. 
\end{proof}

Finally, we can solve this equation:
\begin{lemma} \label{lem:A-solution}
	For any $2\leq m_1<\dots<m_i\leq n$, we have $A^{C_{m_1}\cdots C_{m_i}}_{[2,n],\emptyset}=A^{\id}_{I,J}$, where $J=\{m_1,\dots,m_i\}$.
\end{lemma}

\begin{proof} We prove it by induction on $i$. The base case $i=0$ is trivial. Assume we know it for $i$. Then, for $i+1$ we have:
\begin{align*}
A^{C_{m_1}\cdots C_{m_{i+1}}}_{[2,n],\emptyset} & =A^{C_{m_1}}_{[2,n]\setminus \{m_2,\dots,m_{i+1}\},\{m_2,\dots,m_{i+1}\}} \\
& = \sum_{\substack {j\not\in \{m_2,\dots,m_{i+1}\} \\ k<j } } w_{C_{m_1}(k),C_{m_1}(j)}
- \sum_{\substack {j\in \{m_2,\dots,m_{i+1}\} \\ k<j } } w_{C_{m_1}(k),C_{m_1}(j)}.
\end{align*}
Since $C_{m_1}$ acts only on $1,\dots,m_1$, it doesn't affect the second sum and the part of the first sum for $j>m_1$. Since it is a cycle, the only terms when $k<j$ and $C_{m_1}(k)>C_{m_1}(j)$ hold simultaneously are the terms with $k=1$. Hence this total expression is equal to
\begin{align*}
\sum_{\substack {j\not\in \{m_1,m_2,\dots,m_{i+1}\} \\ k<j } } w_{k,j} - \sum_{\substack {k<m_1 } } w_{k,m_1}
- \sum_{\substack {j\in \{m_2,\dots,m_{i+1}\} \\ k<j } } w_{k,j} = A^{\id}_{[2,n]\setminus \{m_1,\dots,m_{i+1}\},\{m_1,\dots,m_{i+1}\}} .
\end{align*}
\end{proof}

Now we are ready to prove Lemma~\ref{lem:A-A}. The first statement follows from Lemmas~\ref{lem:leftmultA} and~\ref{lem:A-solution}. Then, note that Lemmas~\ref{lem:A-solution} and~\ref{lem:uniqueA} imply that the equality $A^{\sigma}_{[2,n],\emptyset}=A^{\tau}_{I,J}$ can hold only for $\tau^{-1}\sigma=C_{m_1}\cdots C_{m_i}$, where $J=\{m_1<\cdots <m_i\}$ (and $\tau^{-1}\sigma = \id$ if $J=\emptyset$). Hence $\tau=\sigma C_{m_i}^{-1}\cdots C_{m_1}^{-1}$.

\subsubsection{Proof of Lemma~\ref{lem:coeffP}} First, observe that the basic properties of $u_{ij}$ imply the following identity that we'll use in the proof (one can prove it by induction on $r$, for instance):
\begin{equation}\label{eq:u-identity}
\sum_{m = r+1}^{n-1}  \frac{u_{1,r+1} u_{m,m+1}} {u_{m,1} u_{1, m+1}} + \frac{u_{1, r+1}}{u_{n,1}} = -1.
\end{equation}

Second, observe that Equation~\eqref{eq:idenQ} is invariant under the left products with any $\rho\in S_n$, so it is sufficient to prove it for $\sigma=\id$. We, however, prove a more general statement. Namely, for any $1\leq i \leq b \leq n$ we prove that
\[
\sum_{2\leq m_1<\cdots < m_{i-1}\leq b} Q(C_{m_{i-1}}^{-1} \cdots C_{m_1}^{-1})
= \begin{cases}
(-1)^{i-1} Q(\id) & b=n; \\
(-1)^{i-1} Q(\id) \frac{u_{i,b+1}}{u_{1,b+1}} & b < n.
\end{cases}
\]
This can be proved by induction on $i$, with the case $i=1$ being obvious. Assume this statement is proved for $i$. Then for $i+1$ we have (the computation is completely analogous in the cases $b=n$ and $b<n$, so we perform it only in the first case):
\begin{align*}
& \sum_{2\leq m_1<\cdots < m_{i}\leq n} Q( C_{m_{i}}^{-1} \cdots C_{m_1}^{-1}) 
= \sum_{m_{i}=i+1}^n \sum_{2\leq m_1<\cdots < m_{i-1}\leq m_{i}-1} Q( C_{m_{i}}^{-1} C_{m_{i-1}}^{-1}\cdots C_{m_1}^{-1}) \\
& = \sum_{m_{i}=i+1}^n (-1)^{i-1}Q(C_{m_i}^{-1}) \frac{u_{C_{m_i}^{-1}(i),C_{m_i}^{-1}(m_i)}}{u_{C_{m_i}^{-1}(1),C_{m_i}^{-1}(m_i)}} 
= (-1)^{i-1} Q(\id) \left(\sum_{m_{i}=i+1}^{n-1} \frac{u_{1,2}u_{m_i,m_i+1}}{u_{m_i,1}u_{1,m_i+1}}  + \frac{u_{1,2}}{u_{n,1}}\right) \frac{u_{i+1,1}}{u_{2,1}} \\
& = (-1)^{i} Q(\id).
\end{align*}
Here the second equality is the induction assumption, and the final equality follows from Equation~\eqref{eq:u-identity}.

\section{The principal terms} \label{sec:Principal}
Recall a reformulation of the formula for $\mathcal{F}_n^{\mathsf{Ok}}$ proposed in~\cite[Equation (3.3)]{OkounkovMain}:
\begin{align} \label{eq:ok-principal}
\mathcal{F}_n^{\mathsf{Ok}}= \frac{(-1)^{n+1}(2\pi)^{n/2}}{\prod_{j=1}^n\sqrt{x_j}}
\mathcal{E}^{\circlearrowleft}\left(\frac{x_1}{2^{1/3}},\dots,\frac{x_n}{2^{1/3}}\right)
+ \mathrm{diagonal\ terms}.
\end{align}
The idea behind this formula is that the whole expression for $\mathcal{F}_n^{\mathsf{Ok}}$ can be considered as the regularization of its principal part, which is the first summand on the right hand side of Equation~\eqref{eq:ok-principal}, by the terms that are Laplace transforms of distributions supported on the diagonals, see~\cite[Sections 2.6.3 and 3.1.4]{OkounkovMain}. 

The formula of Buryak, in the form of Equation~\eqref{eq:buryak-symmetric}, can also be represented as the sum of its principal part and the regularizing terms supported on the diagonals. Firstly, we interpret the integrals as Cauchy principal values in order to interchange $\int_{\mathbb{R}^n}$ and $\sum_{\sigma\in S_n}$ in Equation~\eqref{eq:buryak-symmetric}. We obtain:
\begin{align}\label{eq:buryak-principal}
\mathcal{F}_n^{\mathsf{Bur}}= & \sum_{\sigma \in S_n}
\frac{e^{ \frac{1}{24}\left( \sum\limits_{j=1}^{n} x_j \right)^3 }}{\left( \sum\limits_{j=1}^{n} x_j \right) (2 \pi)^{\frac n2} \prod_{j=1}^{n}  x_j^{\frac 32} } \int_{\mathbb{R}^n} \left[\prod_{j=1}^{n} e^{ - \frac{a_j^2}{2 x_j} } da_j \right]  \frac{\exp \left( \frac{\I}{2} \sum_{j < k} 
	a_{\sigma(j)}x_{\sigma(k)}-a_{\sigma(k)}x_{\sigma(j)}
	\right)}{\prod_{j=1}^{n-1} \I \left( \frac{a_{\sigma(j)}}{x_{\sigma(j)}} - \frac{a_{\sigma(j+1)}}{x_{\sigma(j+1)}}\right)  } 
\end{align}
Here the expressions under the sign of the integral have poles along the diagonals defined as $a_{\sigma(j)}/x_{\sigma(j)}- a_{\sigma(j+1)}/x_{\sigma(j+1)} = 0$, $j=1,\dots,n-1$. Recall the integrals should be understood as the Cauchy principal value integrals, that is, we exclude the tubular neighborhood of the divisor of poles of the radius $r$, integrate, and take the $r\to 0$ limit of the resulting expression. Similarly to Okounkov's formula, they can be decomposed into a principal part without poles and a diagonal part by applying the Sokhotski-Plemelj formula.

\begin{lemma} \label{lem:MainLemmaPrincipal}
	The right hand side of Equation~\eqref{eq:buryak-principal} decomposes in a similar way to the right hand side of Equation~\eqref{eq:ok-principal}, that is, into a sum of its principal part and some diagonal regularization terms. The principal parts of the right hand sides of  Equations~\eqref{eq:ok-principal} and~\eqref{eq:buryak-principal} are equal. 
\end{lemma}

\begin{proof} Fix $\sigma\in S_n$ and consider the corresponding summand on the right hand of Equation~\eqref{eq:buryak-principal}. We apply the following change of the variables $a_1,\dots,a_n$:
\[
a_{\sigma(j)}  =  b_{\sigma(j)}+\frac \I 2 x_{\sigma(j)} \left(-\sum_{\ell<j} x_{\sigma(\ell)} + \sum_{r>j} x_{\sigma(r)}\right)
\]
With this change of variables we have:
\[
\frac 18\sum_{j\not=k} x_jx_k^2 + \frac{1}{4} \sum_{j<k<t} x_jx_kx_t +\frac \I 2\sum_{k<t} \left(a_{\sigma(k)}x_{\sigma(t)} -a_{\sigma(t)}x_{\sigma(k)}\right) - \sum_{j=1}^n \frac{a_j^2}{2x_j}
=-\sum_{j=1}^n \frac{b_j^2}{2x_j},
\]
and 
\[
\I \left(\frac{a_{\sigma(j)}}{x_{\sigma(j)}} -\frac{a_{\sigma(j+1)}}{x_{\sigma(j+1)}}\right) = \I \left(\frac{b_{\sigma(j)}}{x_{\sigma(j)}} -\frac{b_{\sigma(j+1)}}{x_{\sigma(j+1)}}\right) - \frac{\left(x_{\sigma(j)}+x_{\sigma(j+1)}\right)}{2}.
\]
Thus, the right hand side of Equation~\eqref{eq:buryak-principal} is equal to
\begin{align} \label{eq:BuryakPrincSymmetric}
\sum_{\sigma \in S_n}
\frac{e^{ \frac{1}{24}\sum\limits_{j=1}^{n} x_j^3 }}{\left( \sum\limits_{j=1}^{n} x_j \right) (2 \pi)^{\frac n2} \prod_{j=1}^{n}  x_j^{\frac 32} } \int_{\mathbb{R}^n}  \frac{\prod_{j=1}^{n} e^{ - \frac{b_j^2}{2 x_j} } db_j  }{\prod_{j=1}^{n-1} \left[\I \left(\frac{b_{\sigma(j)}}{x_{\sigma(j)}} -\frac{b_{\sigma(j+1)}}{x_{\sigma(j+1)}}\right) - \frac{\left(x_{\sigma(j)}+x_{\sigma(j+1)}\right)}{2} \right] } & + \mathrm{diagonal\ terms} \\ \notag
= - \sum_{\sigma \in S_n/\mathbb{Z}_n}
\frac{e^{ \frac{1}{24}\sum\limits_{j=1}^{n} x_j^3 }}{(2 \pi)^{\frac n2} \prod_{j=1}^{n}  x_j^{\frac 32} } \int_{\mathbb{R}^n}  \frac{\prod_{j=1}^{n} e^{ - \frac{b_j^2}{2 x_j} } db_j  }{\prod_{j=1}^{n} \left[\I \left(\frac{b_{\sigma(j)}}{x_{\sigma(j)}} -\frac{b_{\sigma(j+1)}}{x_{\sigma(j+1)}}\right) - \frac{\left(x_{\sigma(j)}+x_{\sigma(j+1)}\right)}{2} \right] } & + \mathrm{diagonal\ terms},
\end{align}
where in the second line $\sigma(n+1)$ denotes $\sigma(1)$. The diagonal terms are half-residues arising as a result of translating the contour of the $b_{\sigma(k)}$'s back to $\mathbb{R}^n$, removing the diagonal singularities in the process. An explicit expression for the diagonal terms will be computed in the next section using the Sokhotski-Plemelj formula.

\begin{remark}
Let us note that Equation (\ref{eq:BuryakPrincSymmetric}) is similar to the expressions for the $n$-point functions obtained by Br\'{e}zin and Hikami in \cite{BH0704,BH0709}.
\end{remark}

Since we got a sum over $\sigma\in S_n/\mathbb{Z}_n$, as in the principal part of the right hand side of Equation~\eqref{eq:ok-principal}, it is sufficient to prove for each $\sigma\in S_n/\mathbb{Z}_n$ that the corresponding summands are equal. Without loss of generality we can assume that $\sigma = [\mathrm{id}]$. Then we have to prove that
\begin{align} \label{eq:OkounBuryakPrincipal}
& -\frac{e^{ \frac{1}{24}\sum\limits_{j=1}^{n} x_j^3 }}{(2 \pi)^{\frac n2} \prod_{j=1}^{n}  x_j^{\frac 32} } \int_{\mathbb{R}^n}  \frac{\prod_{j=1}^{n} e^{ - \frac{b_j^2}{2 x_j} } db_j  }{\prod_{j=1}^{n} \left[\I \left(\frac{b_{j}}{x_{j}} -\frac{b_{j+1}}{x_{j+1}}\right) - \frac{\left(x_{j}+x_{j+1}\right)}{2} \right] } \\ \notag
& = \frac{(-1)^{n+1}(2\pi)^{n/2}}{\prod_{j=1}^n\sqrt{x_j}}\frac{e^{ \frac{1}{12}\sum\limits_{j=1}^{n} \left(\frac{x_j}{2^{1/3}}\right)^3 } }{\prod_{j=1}^{n} \sqrt{4\pi \left(\frac{x_j}{2^{1/3}}\right)} } \int_{\mathbb{R}_{\geq 0}^n} \left[\prod_{j=1}^{n} ds_j \right]
\exp \left({-\sum_{j=1}^n \left(\frac{(s_j-s_{j+1})^2}{4\left(\frac{x_j}{2^{1/3}}\right)} + \frac{(s_j+s_{j+1})x_j}{2^{4/3}} \right)}\right),
\end{align}
or, equivalently, if we cancel the common factors and rescale $s_j$ by $2^{-1/3}$, we have to prove that
\begin{align} \label{eq:laststep}
& \frac{1}{\prod\limits_{j=1}^{n} (2\pi x_j)^{\frac 12} } \int_{\mathbb{R}^n}  \frac{\prod\limits_{j=1}^{n} e^{ - \frac{b_j^2}{2 x_j} } db_j  }{\prod\limits_{j=1}^{n} \left[-\I \left(\frac{b_{j}}{x_{j}} +\frac{b_{j+1}}{x_{j+1}}\right) + \frac{\left(x_{j}+x_{j+1}\right)}{2} \right] } \\ \notag
& = \int_{\mathbb{R}_{\geq 0}^n} \left[\prod_{j=1}^{n} ds_j \right]
\exp \left(
-\sum_{j=1}^n 
\left(
	\frac{(s_j-s_{j+1})^2}{2x_j} + \frac{(s_j+s_{j+1})x_j}{2} 
\right)
\right).
\end{align}
To this end, we use the following trick. Replace $\left[-\I \left(\frac{b_{j}}{x_{j}} +\frac{b_{j+1}}{x_{j+1}}\right) + \frac{\left(x_{j}+x_{j+1}\right)}{2} \right]^{-1}$ by 
\[
\int_{\mathbb{R}_{\geq 0}} ds_{j+1} \exp\left( s_{j+1} \left[\I \left(\frac{b_{j}}{x_{j}} +\frac{b_{j+1}}{x_{j+1}}\right) - \frac{\left(x_{j}+x_{j+1}\right)}{2} \right] \right),
\]
where $s_{n+1}$ denotes $s_1$, change the order of integration and take the Gaussian average with respect to the variables $b_j$. We see that the left hand side of the equation~\eqref{eq:laststep} is equal to
\begin{align} 
& \int_{\mathbb{R}_{\geq 0}^n} \left[\prod_{j=1}^{n} ds_j \right] \int_{\mathbb{R}^n} \frac{\prod\limits_{j=1}^{n}db_j}{\prod\limits_{j=1}^{n} (2\pi x_j)^{\frac 12} } 
\prod\limits_{j=1}^{n} \exp \left(
-\frac{1}{2x_j}\left(b_j^2 -\I 2 a_j s_{j+1} +\I 2a_j s_j\right) - \frac{(x_j+x_{j+1})s_{j+1}}{2} 
\right)
\\ \notag
& = \int_{\mathbb{R}_{\geq 0}^n} \left[\prod_{j=1}^{n} ds_j \right]
\exp \left(
\sum_{j=1}^n 
\left(
\frac{(\I s_j-\I s_{j+1})^2}{2x_j} - \frac{(s_j+s_{j+1})x_j}{2} 
\right)
\right),
\end{align}
which is precisely the right hand side of Equation~\eqref{eq:laststep}. This computation proves Equation~\eqref{eq:laststep}, and, therefore, completes the proof of the lemma.
\end{proof}

\begin{remark}\label{rem:OkounkovArgument} The argument of Okounkov in~\cite[Section 3.1]{OkounkovMain} implies that it is sufficient to compare the principal terms of $\mathcal{F}_n^{\mathsf{Bur}}$ and  $\mathcal{F}_n^{\mathsf{Ok}}$ in order to prove the coincidence of these formula, since the diagonal terms only compensate for the non-regular terms in the principal part detected by the wrong powers of $\pi$ (it is also the case for $\mathcal{F}_n^{\mathsf{Bur}}$, where this property is evident from the Sokhotski–Plemelj formula). So, Lemma~\ref{lem:MainLemmaPrincipal} implies Theorem~\ref{thm:main}. However, we can explicitly identify the diagonal terms in $\mathcal{F}_n^{\mathsf{Bur}}$ and  $\mathcal{F}_n^{\mathsf{Ok}}$, and we do this in the next section.
\end{remark}

\section{Diagonal contributions} \label{sec:diagonal}

We represent Buryak's formula in the following way. 

\begin{theorem}\label{thm:fullcorrespondence} (1) We have:
	\begin{align}\label{eq:BuryakCyclicDiagonal}
	\mathcal{F}_n^{\mathsf{Bur}}= \frac{(2\pi)^{\frac n2}}{\prod\limits_{j=1}^{n}  x_j^{\frac 12}}& \sum_{\ell=1}^{n} \sum_{\substack{[I_1\sqcup\cdots\sqcup I_\ell ]\\ = \{1,\dots,n\}}} \frac{-e^{ \frac{1}{24}\sum\limits_{j=1}^{\ell} x_{I_j}^3  }}{ (2 \pi)^{\ell}  }\int_{\mathbb{R}^\ell}    \frac{\prod_{j=1}^{\ell} e^{ - \frac{f_j^2}{2 x_{I_j}} } \frac{df_j}{x_{I_j}} }{\prod_{j=1}^{\ell} \left[\I \left(\frac{f_{j}}{x_{I_j}} - \frac{f_{j+1}}{x_{I_{j+1}}}\right) - \frac{x_{I_j}+x_{I_{j+1}}}{2} \right]  } . 
	\end{align}
Here we take the sum over the cyclicly ordered partitions of $\{1,\dots,n\}$, that is, $[I_1\sqcup \cdots \sqcup I_\ell]$ is identified with $[I_2\sqcup \cdots \sqcup I_\ell\sqcup I_1]$, and $I_{\ell+1}$ denotes $I_1$ and $f_{\ell+1}$ denotes $f_1$. 

(2) For every cyclicly ordered partitions of $\{1,\dots,n\}$, $[I_1\sqcup \cdots \sqcup I_\ell]$, we have:
\begin{align}\label{eq:BuryakOkounkovFullcorrespondence}
\frac{e^{ \frac{1}{24}\sum\limits_{j=1}^{\ell} x_{I_j}^3}}{ (2 \pi)^{\ell}  }
\int_{\mathbb{R}^\ell}    \frac{\prod_{j=1}^{\ell} e^{ - \frac{f_j^2}{2 x_{I_j}} } \frac{df_j}{x_{I_j}} }{\prod_{j=1}^{\ell} \left[\I \left(\frac{f_{j}}{x_{I_j}} - \frac{f_{j+1}}{x_{I_{j+1}}}\right) - \frac{x_{I_j}+x_{I_{j+1}}}{2} \right] } =(-1)^\ell\mathcal{E}\left(\frac{x_{I_1}}{2^{1/3}},\dots, \frac{x_{I_\ell}}{2^{1/3}}\right).
\end{align}
\end{theorem}

This theorem is a refinement of Lemma~\ref{lem:MainLemmaPrincipal} that includes now all the diagonal terms and we have an explicit term-wise identification. It immediately implies Theorem~\ref{thm:main}. We devote the rest of this section to the proof of Theorem~\ref{thm:fullcorrespondence}, whose main part consists of a careful application of the Sokhotski-Plemelj formula, and the further steps just repeat the computations in the proof of Lemma~\ref{lem:MainLemmaPrincipal}.

\subsection{Structure of the Sokhotski–Plemelj formula}
Let us discuss explicitly how to apply the Sokhotski–Plemelj formula to Equation~\eqref{eq:buryak-symmetric}. In principle, one can just directly iteratively apply it, but we first discuss the structure of the formula since it simplifies computation a lot.

Fix a particular $\sigma\in S_n$ and consider the corresponding term in the variables
\[
f = \frac {\sum_{i=1}^n x_i}n \left(\sum_{i=1}^n\frac{a_i}{x_i} \right), \qquad g_{i} = \frac{a_{\sigma(i)}}{x_{\sigma(i)}} - \frac{a_{\sigma(i+1)}}{x_{\sigma(i+1)}}, \quad i=1,\dots,n-1. 
\]
In these variables the shift of $a_i$'s that we applied in the previous section looks like
\[
f \to f-\frac{\I \sum_{i=1}^n x_i}{2n}\sum_{i=1}^n (2i-1-n) x_{\sigma(i)}, \qquad g_{i} \to g_i - \frac{\I}{2} (x_{\sigma(i)}+x_{\sigma(i+1)}), \quad i=1,\dots,n-1. 
\]
The denominator of the expression under the integral is equal to $g_1\cdots g_{n-1}$. Since there is no pole in $f$, its shift is neglectable. Assuming $x_1,\dots,x_n$ to be small positive real numbers, we move the contour of integration for each $g_i$ to the lower half-plane, and then can deform it back to the real line with excluded interval around $g_i=0$ and a half-circle around it in the lower half-plane, which in the limit gives the sum of the principal part and the half-residue at $g_i=0$. This is exactly the Sokhotski–Plemelj formula applied now to the product of simple poles $g_1\cdots g_{n-1}$. 

The whole integral expression is then split into $2^{n-1}$ summands for each $\sigma$, since we have to make a choice for each $g_i$ whether we take the principal part or the residue part of its contour.  If we choose for all $g_i$'s the principal part of the integral, we exactly obtain the principal terms considered in the previous section. More generally, the full system of choices is controlled by pairs $(\sigma, \sqcup_{i=1}^\ell I_i)$, where $\sigma\in S_n$ and $\sqcup_{i=1}^\ell I_i=\{1,\dots,n\}$ and $I_1<\cdots <I_\ell$ in the sense that for any $n_j\in I_{i_j}$, $j=1,2$, $i_1<i_2$ implies $n_1<n_2$. Once we fixed a pair $(\sigma, \sqcup_{i=1}^\ell I_i)$, we choose the residue option for all $g_{\sigma(i)}$'s with $i \in I_j \setminus \{ \max(I_j)\}$, $j=1,\dots,\ell$, and the principal option for all $g_{\sigma(i)}$'s with $i=\max(I_j)$, $j=1,\dots,\ell-1$. 

Note that the integrals for the pairs $(\sigma\rho, \sqcup_{i=1}^\ell I_i)$, $\rho(I_i)=I_i$ for $i=1,\dots,\ell$, coincide. Moreover, each of them contributes $\prod_{i=1}^\ell 1/|I_i|!$ to the product of negative residues since the contour of integration in the plane $\sum_{j\in I_i} a_{\sigma(j)}/x_{\sigma(j)}=0$, $i=1,\dots,\ell$, is the intersection of the torus around the origin with the Weyl chamber selected by the inequalities $a_{\sigma\rho(j_1)}/x_{\sigma\rho(j_1)} <a_{\sigma\rho(j_2)}/x_{\sigma\rho(j_2)} $ for $j_1,j_2\in I_i$, $j_1<j_2$. Thus the residue part of the integral in the sum over all $\rho\in S_n$ such that $\rho(I_i)=I_i$ for $i=1,\dots,\ell$ is the product of the full residues around zero in the planes $\sum_{j\in I_i} a_{\sigma(j)}/x_{\sigma(j)}=0$, $i=1,\dots,\ell$, with the coefficient $\prod_{i=1}^\ell (2\pi\I)^{|I_i|-1}$.

Now we are ready to perform the computation. For simplicity we take $\sigma=\id$, $\ell=1$, and $I_1=\{1,\dots,n\}$, and treat the general case as an $\ell$-fold iteration of the same computation, with the indices adjusted with respect to a general $\sigma$. 

\subsubsection{Computation for $(\id,\{1,\dots,n\})$} In the case $\sigma=\id$ and $\ell=1$, $I_1=\{1,\dots,n\}$, we take the sum over all $\rho\in S_n$. The corresponding residue term  is equal to
\[
\frac{e^{ \frac{1}{24} \left( \sum\limits_{j=1}^{n} x_j \right)^3 }}{\left( \sum\limits_{j=1}^{n} x_j \right) (2 \pi)^{\frac n2} \prod_{j=1}^{n}  x_j^{\frac 12} }\int_{\mathbb{R}} \oint_{(S^1)^{n-1}} \left[\prod_{j=1}^{n} e^{ - \frac{a_j^2}{2 x_j} } \frac{da_j}{x_j} \right]  \frac{\exp \left( \frac{\I}{2} \sum_{j < k} 
	a_{j}x_{k}-a_{k}x_{j}
	\right)}{\prod_{j=1}^{n-1} \I \left( \frac{a_{j}}{x_{j}} - \frac{a_{j+1}}{x_{j+1}}\right)  } . 
\]
Note that $- \sum_{j=1}^n \frac{a_j^2}{2 x_j} = -f^2/2\left(\sum_{j=1}^n x_j\right) + O(g_1,\dots,g_{n-1})$, $\prod_{j=1}^n da_j/x_j = \prod_{j=1}^{n-1} dg_j df/\left(\sum_{j=1}^n x_j\right)$, and $\sum_{j < k} 
a_{j}x_{k}-a_{k}x_{j} = O(g_1,\dots,g_{n-1})$. This allows us to rewrite the residue as 
\[
(2\pi)^{n-1} \frac{e^{ \frac{1}{24}\left( \sum\limits_{j=1}^{n} x_j \right)^3 }}{\left( \sum\limits_{j=1}^{n} x_j \right) (2 \pi)^{\frac n2} \prod_{j=1}^{n}  x_j^{\frac 12} } \int_{\mathbb{R}}  e^{ - f^2/\left(2 \sum_{j=1}^n x_j \right) } \frac{df}{\sum\limits_{j=1}^n x_j }.
\]

\subsubsection{Computation for $\sigma=\id$, general partition} Recall that we denote by $x_I$, $I\subset \{1,\dots,n\}$ the sum $\sum_{i\in I} x_I$. In the case of a general partition $\sqcup_{i=1}^\ell I_i$, it is more convenient to work in the coordinates 
\[
f_i = \frac {x_{I_i}}{|I_i|} \left(\sum_{j\in I_i}\frac{a_i}{x_i} \right), \quad g_{ij} = \frac{a_{j}}{x_{j}} - \frac{a_{j+1}}{x_{j+1}}, \quad i=1,\dots,\ell, \quad j\in I_i\setminus \{\max(I_i)\}. 
\]
The corresponding residue term is equal to the principal part of 
\[
\frac{e^{ \frac{1}{24}\left( \sum\limits_{j=1}^{n} x_j \right)^3 }}{\left( \sum\limits_{j=1}^{n} x_j \right) (2 \pi)^{\frac n2} \prod_{j=1}^{n}  x_j^{\frac 12} }\int_{\mathbb{R}^\ell} \oint_{(S^1)^{n-\ell}} \left[\prod_{j=1}^{n} e^{ - \frac{a_j^2}{2 x_j} } \frac{da_j}{x_j} \right]  \frac{\exp \left( \frac{\I}{2} \sum_{j < k} 
	a_{j}x_{k}-a_{k}x_{j}
	\right)}{\prod_{j=1}^{n-1} \I \left( \frac{a_{j}}{x_{j}} - \frac{a_{j+1}}{x_{j+1}}\right)  } , 
\]
where the integral over $\mathbb{R}^\ell$ is the Cauchy principle value integral (except for the diagonal direction, where it is a converging integral).
In the new coordinates, we have: 
\begin{align*}
 & - \sum_{j=1}^n \frac{a_j^2}{2 x_j}  = - \sum_{i=1}^\ell \frac{f_i^2}{2 x_{I_i}} + O(g_{ij}) ; \qquad
   \prod_{j=1}^{n} \frac{da_j}{x_j} = \prod_{i=1}^{\ell} \prod_{j=1}^{|I_i|-1} dg_{ij} \prod_{i=1}^\ell \frac{df_i}{x_{I_i}} ; \\
& \prod_{i=1}^{\ell-1} \I \left( \frac{a_{\max(I_i)}}{x_{\max(I_i)}} - \frac{a_{\max(I_i)+1}}{x_{\max(I_i)+1}}\right)  = \prod_{i=1}^{\ell-1} \I \left( \frac{f_i}{x_{I_i}} - \frac{f_{i+1}}{x_{I_{i+1}}}\right) + O(g_{ij}); \\
& \sum_{j < k} 
a_{j}x_{k}-a_{k}x_{j} = \sum_{j < k} 
f_{j}x_{I_k}-f_{k}x_{I_j}+ O(g_{ij}) .
\end{align*}
This allows us to rewrite the residue formula as 
\begin{equation}\label{eq:diagonalsI}
\prod_{i=1}^\ell (2\pi)^{|I_i|-1} \frac{e^{ \frac{1}{24}\left( \sum\limits_{j=1}^{n} x_j \right)^3 }}{\left( \sum\limits_{j=1}^{n} x_j \right) (2 \pi)^{\frac n2} \prod_{j=1}^{n}  x_j^{\frac 12} }\int_{\mathbb{R}^\ell}  \left[\prod_{j=1}^{\ell} e^{ - \frac{f_j^2}{2 x_{I_j}} } \frac{df_j}{x_{I_j}} \right]  \frac{\exp \left( \frac{\I}{2} \sum_{j < k} 
	f_{j}x_{I_k}-f_{k}x_{I_j}
	\right)}{\prod_{j=1}^{\ell-1} \I \left( \frac{f_{j}}{x_{I_j}} - \frac{f_{j+1}}{x_{I_{j+1}}}\right)  } . 
\end{equation}
The diagonal terms of this expression will be transferred to the partitions of $\{1,\dots,n\}$ with $\ell' < \ell $ terms, so we have to take the principal part:
\begin{equation}\label{eq:diagonalsIprincipal}
\prod_{i=1}^\ell (2\pi)^{|I_i|-1} \frac{e^{ \frac{1}{24} \sum\limits_{j=1}^{\ell} x_{I_j} ^3 }}{\left( \sum\limits_{j=1}^{n} x_j \right) (2 \pi)^{\frac n2} \prod\limits_{j=1}^{n}  x_j^{\frac 12} }\int_{\mathbb{R}^\ell}  \left[\prod_{j=1}^{\ell} e^{ - \frac{f_j^2}{2 x_{I_j}} } \frac{df_j}{x_{I_j}} \right]  \frac{1}{\prod\limits_{j=1}^{\ell-1} \left[ \I \left( \frac{f_{j}}{x_{I_j}} - \frac{f_{j+1}}{x_{I_{j+1}}}\right) - \frac{x_{I_j} + x_{I_{j+1}}}{2} \right] } . 
\end{equation}

\subsubsection{General $\sigma$, general partition} If we have a general $\sigma$, it just means that we no longer have to assume that the subsets $I_1,\dots,I_\ell$ satisfy the property that for any $n_j\in I_{i_j}$, $j=1,2$, $i_1<i_2$ implies $n_1<n_2$. That is, we obtain the same formula as Equation~\eqref{eq:diagonalsIprincipal}, with arbitrary ordered sequence $I_1,\dots,I_\ell$ such that $\sqcup_{i=1}^\ell I_i = \{1,\dots,n\}$. 
We have:
\begin{align}\label{eq:BuryakPrincipalValue}
\mathcal{F}_n^{\mathsf{Bur}}=\sum_{\ell=1}^{n} \sum_{\substack{I_1\sqcup\cdots\sqcup I_\ell \\ = \{1,\dots,n\}}} \frac{(2\pi)^{\frac n2}}{\prod\limits_{j=1}^{n}  x_j^{\frac 12}}& \frac{e^{ \frac{1}{24} \sum\limits_{j=1}^{\ell} x_{I_j}^3 }}{\left( \sum\limits_{j=1}^{n} x_j \right) (2 \pi)^{\ell}  }\int_{\mathbb{R}^\ell}  \left[\prod_{j=1}^{\ell} e^{ - \frac{f_j^2}{2 x_{I_j}} } \frac{df_j}{x_{I_j}} \right]  \frac{1}{\prod\limits_{j=1}^{\ell-1} \left[ \I \left( \frac{f_{j}}{x_{I_j}} - \frac{f_{j+1}}{x_{I_{j+1}}}\right) - \frac{x_{I_j} + x_{I_{j+1}}}{2} \right] } . 
\end{align}
The theorem follows by direct comparison of this expression with the principal part of $\mathcal{F}^{\textrm{Bur}}_l (x_{I_1}, \dots, x_{I_l})$ which, by section 4, is related to $\mathcal{E}\left( \frac{x_{I_1}}{2^{1/3}}, \dots, \frac{x_{I_\ell}}{2^{1/3}}  \right)$.

\subsubsection{Final remarks} We relate Equation~\eqref{eq:BuryakCyclicDiagonal} and~\eqref{eq:BuryakPrincipalValue} exactly in the same way as the two sides of Equation~\eqref{eq:BuryakPrincSymmetric}, see the first half of the proof of Lemma~\ref{lem:MainLemmaPrincipal}. The proof of Equation~\eqref{eq:BuryakOkounkovFullcorrespondence} repeats literally the proof of Equation~\eqref{eq:OkounBuryakPrincipal}, see the second half of the proof of Lemma~\ref{lem:MainLemmaPrincipal}, one just have to replace the symbols $n$, $(b_1,\dots,b_n)$, and $(x_1,\dots,x_n)$ in that argument by $\ell$, $(f_1,\dots,f_\ell)$, and $(x_{I_1},\dots,x_{I_\ell})$.

\bibliographystyle{alpha}
\bibliography{buryakokounkovref}

\end{document}